\documentclass[10pt]{amsart}

\usepackage{amssymb,amsmath,amsthm,amsfonts,microtype}
\usepackage{pdfsync}

\newcommand{\comment}[1]{}

\newtheorem{lem}{Lemma}%[section]
\newtheorem{propn}{Proposition}%[section]
\newtheorem{cor}{Corollary}%[section]
\newtheorem{thm}{Theorem}%[section]

\newtheorem*{thm2}{Theorem 2}

\theoremstyle{remark}

\theoremstyle{definition}
\newtheorem{defn}{Definition}[section]

\newcommand{\R}{\mathbb R}

\newcommand{\N}{\mathbb N}

\DeclareMathOperator{\supp}{supp}

\newcommand{\vp}{\varphi}
\newcommand{\D}{\delta}

\newcommand{\VE}{\varepsilon}
\newcommand{\A}{\alpha}

\newcommand{\lm}{\lambda}

\newcommand{\subs}{\subseteq}

\newcommand{\be}{\begin{equation}}
\newcommand{\ee}{\end{equation}}
\newcommand{\bee}{\begin{equation*}}
\newcommand{\eee}{\end{equation*}}

\begin{document}
\title{Distance Graphs and sets of positive upper density in $\R^d$}
%\title{Rectangles and Product of Simplices \\ in sets of positive upper density in $\R^d$}
%\title{Rectangles and sets of positive upper density in $\R^d$}
%Embedding squares into sets of positive upper density in $\R^4$}
%{Squares are Density Ramsey}
\author{Neil Lyall\quad\quad\quad\'Akos Magyar}
\thanks{The first author was partially supported by grant NSF-DMS 1702411 and the Simons Foundation Collaboration Grant for Mathematicians 245792. The second author was partially supported by grant NSF-DMS 1600840 and ERC-AdG 321104.}

\address{Department of Mathematics, The University of Georgia, Athens, GA 30602, USA}
\email{lyall@math.uga.edu}
\address{Department of Mathematics, The University of Georgia, Athens, GA 30602, USA}
\email{magyar@math.uga.edu}


%\keywords{}

\begin{abstract}
We present a sharp extension of a result of Bourgain on finding configurations of $k+1$ points in general position in measurable subset of $\R^d$ of positive upper density whenever $d\geq k+1$ to all proper $k$-degenerate distance graphs.
 \end{abstract}
%\date{\today}
\maketitle
%\tableofcontents

\setlength{\parskip}{5pt}
%\setlength{\parindent}{0pt}

%%%%%%%%%%%%%%%%%%%%%%%%%%%%%%%%%%%%%%%%%%%%%%%

\section{Introduction}\label{intro}

\subsection{Background}

A result of Katznelson and Weiss \cite{FKW} states that
if $A\subseteq\mathbb{R}^2$ has positive upper Banach density, then its distance set
%\[\text{dist}(A)=\{|x-x'|\,:\, x,x'\in A\}\]
 $\{|x-x'|\,:\, x,x'\in A\}$ contains all large numbers.
%The definition of upper Banach density is recalled in Section \ref{Reduction} below.
Recall that the \emph{upper Banach density} of a measurable set $A\subseteq\R^d$ is defined by
\be\label{BD}
\D^*(A)=\lim_{N\rightarrow\infty}\sup_{t\in\R^d}\frac{|A\cap(t+Q_N)|}{|Q_N|},\ee
where $|\cdot|$ denotes Lebesgue measure on $\R^d$ and $Q_N$ denotes the cube $[-N/2,N/2]^d$.

Note that the distance set any set of positive Lebesgue measure in $\R^d$ automatically contains all sufficiently small numbers (by the Lebesgue density theorem) and  that it is easy to construct a set of positive upper density which does not contain a fixed distance  by  placing small balls centered on an appropriate square grid.

This result was later reproved using Fourier analytic techniques by Bourgain in \cite{B} where he established the following more general result for finite point configurations  in general position.

\begin{thm}[Bourgain \cite{B}]\label{BourSimp}
Let $\Delta_k\subseteq\R^{d}$ be a fixed collection of $(k+1)$ points in general position.

If $A\subseteq\R^d$ has positive upper Banach density and $d\geq k+1$, then
 there exists a threshold $\lm_0=\lm_0(A,\Delta_k)$ such that $A$ contains an isometric copy of $\lm\cdot\Delta_k$ for all $\lambda\geq \lm_0$.
\end{thm}

Recall that a point configuration $\Delta_k'$ is said to be an isometric copy of $\lm\cdot \Delta_k$
% in $\R^d$ with $d\geq k+1$ if $\Delta'_k=x+\lm\cdot U(\Delta_k)$ for  some $x\in \R^d$ and $U\in SO(d)$ when $d\geq k+1$, or equivalently
if there exists a bijection $\phi:\Delta_k\to \Delta_k'$ such that $|\phi(v)-\phi(w)|=\lm\,|v-w|$ for all $v,w\in\Delta_k.$

%\smallskip

These results may be viewed as initial results in \emph{geometric Ramsey theory} where, roughly speaking, one shows that ``large" but otherwise arbitrary sets necessarily contain certain geometric configurations. Recently there has been a number of results in this direction in various contexts, see \cite{BU}, \cite{IO}, and \cite{IP}. One of the aims of this article is to present a common extension for measurable subsets of Euclidean spaces of positive upper density. At the same time we present a new approach, based on a simple notion of uniform distribution attached to an appropriate scale. For another instance of this new approach see \cite{Product} where configurations of points that form the vertices of a rigid geometric square, and more generally the direct product of any two finite point configurations in general position, are addressed.

%The purpose of this article is to establish an extension of Bourgain's result to \emph{proper distance graphs}.

\subsection{Distance Graphs and Main Result}

A \emph{distance graph} $\Gamma=\Gamma(V,E)$ is a connected finite graph with vertex set $V$ contained in $\mathbb{R}^d$ for some $d\geq1$.
%Clearly any $k$-degenerate graph is automatically $\ell$-degenerate for any $\ell>k$, but when referring to a given graph $\Gamma$ as $k$-degenerate will implicitly assume in this article that $k$ has been chosen minimally.
We say that  $\Gamma$ is \emph{$k$-degenerate} if each of its subgraphs  contain a vertex with degree at most $k$ and refer to the smallest $k$ such that it is $k$-degenerate as the \emph{degeneracy} of $\Gamma$.
%A graph is $k$-degenerate if its vertices can be successively deleted so that when deleted, each has degree at most $k$.
It is easy to see that if a given graph is $k$-degenerate, then there exists
% is easily seen to be equivalent to the existence of
an ordering of its vertex set $V=\{v_0,v_1,\dots,v_n\}$ in such a way that $|V_j|\leq k$ for all $1\leq j\leq n$, where
\be V_j:=\{v_i\,:\, (v_i,v_j)\in E \ \text{with} \ 0\leq i<j\}\ee
denotes the set of predecessors of the vertex $v_j$. In this article we shall always assume that the vertices of any given $k$-degenerate graph have been ordered as such.
%To be clear, when we refer to a distance graph $\Gamma$ as $k$-degenerate we
Finally, we shall refer to a distance graph as \emph{proper} if for every $1\leq j\leq n$, the set of vertices $v_j\cup V_j$, namely $v_j$ together with its predecessors, are in general position.

Given a distance graph $\Gamma=\Gamma(V,E)$ and  $\lm>0$ we will say that  $\Gamma'=\Gamma'(V',E')$ is \emph{isometric} to $\lm\cdot\Gamma$ %, and write $\Gamma\simeq\Gamma'$,
if there exists a bijection $\phi:V\to V'$ such that $(v,w)\in E$ if and only if $(\phi(v),\phi(w))\in E'$ and  $|\phi(v)-\phi(w)|=\lm |v-w|$.

\smallskip

The main result of this article is the following

\begin{thm}\label{infinite}
Let $\Gamma=\Gamma(V,E)$ be a proper $k$-degenerate distance graph and $d\geq k+1$.
\begin{itemize}
\item[(i)]
If $A\subseteq\R^d$ has positive upper Banach density, then
 there exists  $\lm_0=\lm_0(A,\Gamma)$ such that $A$ contains an isometric copy of $\lm\cdot\Gamma$ for all $\lambda\geq \lm_0$.
\item[(ii)]
If $A\subseteq[0,1]^d$ with $|A|>0$, then $A$ will contain an isometric copy of $\lm\cdot\Gamma$ for all $\lambda$ in some interval of length at least $\exp(-C_\Gamma |A|^{-C|V|})$.
\end{itemize}
\end{thm}

Note that in both parts of Theorem \ref{infinite} the  dimension $d$ is restricted only by the ``level of degeneracy" of the given distance graph  and not on the number of its vertices which could in fact be arbitrarily large. It is important to further observe that the length of the interval of dilations guaranteed by Part (ii) of Theorem \ref{infinite} depends only on the measure of $A$ and not on the set $A$ itself.

\comment{
\begin{thm}\label{infinite}
Let $\Gamma$ be a proper $k$-degenerate distance graph.
If $A\subseteq\R^d$ has positive upper Banach density and $d\geq k+1$, then
 there exists  $\lm_0=\lm_0(A,\Gamma)$ such that $A$ contains an isometric copy of $\lm\cdot\Gamma$ for all $\lambda\geq \lm_0$.
\end{thm}

\comment{
Recall that the \emph{upper Banach density} of a measurable set $A\subseteq\R^d$ is defined by
\be\label{BD}
\D^*(A)=\lim_{N\rightarrow\infty}\sup_{t\in\R^d}\frac{|A\cap(t+Q_N)|}{|Q_N|},\ee
where $|\cdot|$ denotes Lebesgue measure on $\R^d$ and $Q_N$ denotes the cube $[-N/2,N/2]^d$.
}

%\bigskip

\begin{thm}\label{compact}
Let $\Gamma$ be a proper $k$-degenerate distance graph and $0<\A<1$.
%There exists $c(\A,\Gamma)>0$ such that any
If $A\subseteq[0,1]^d$ with $d\geq k+1$
% and $|A|=\A$
, then $A$ will contain an isometric copy of $\lm\cdot\Gamma$ for all $\lambda$ in some interval of length at least $\exp(-C_\Gamma |A|^{-Cn})$.
%$c(\A,\Gamma)$.
\end{thm}
}

%In fact $c(\A,\Gamma)\geq e^{-c_\Gamma \A^{cn}}$.

Intuitively one should visualize a distance graph with edges made of rigid rods which can freely turn around the vertices, and an isometric embedding of a distance graph into a set $A\subseteq\R^{d}$ as a folding of the graph  so that all of its vertices are supported on $A$.

As mentioned above various special cases of our main result have been established, albeit in different contexts. Indeed, in \cite{BU} the embedding of large copies of trees ($1$-degenerate distance graphs) was shown for dense subsets of the integer lattice. In \cite{IO} it was shown that measurable subsets $A\subs [0,1]^d$ of Hausdorff dimension larger than $\frac{d+1}{2}$ contain an isometric copy of $\lambda\cdot\Gamma$ for all $\lambda$ in some interval, in the special case when $\Gamma$ is a finite path. Very recently, parallel to our work, embedding of bounded degree distance graphs was addressed for subsets of vector spaces over finite fields \cite{IP}.

\medskip

\noindent
\underline{{\bf Examples of Distance Graphs}}

\smallskip

\begin{enumerate}

\item[1.] A non-empty connected graph is $1$-degenerate if and only if it is a tree (contains no cycles). Any tree with vertices in $\R^d$ with $d\geq1$ is isometric to  a proper $1$-degenerate distance graph in $\R^2$.

\smallskip

\item[2.] Cycles with vertices in $\R^d$ with $d\geq1$ form $2$-degenerate distance graphs, but these are not necessarily isometric to a proper $2$-degenerate distance graph in $\R^d$ for any $d\geq1$. Indeed, if $V=\{0,1,2\}\subseteq\R$ and $E=\{(0,1),(1,2),(0,2)\}$, then this defines just such a distance graph.

\smallskip

\item[3.]
If $V=\{(i,j)\,:\,0\leq i,j\leq n\}\subseteq\R^2$ and $E=\{((i,j),(i'j'))\,:\, |i-i'|+|j-j'|=1\}$, then this ``$2$-dimensional grid" forms a proper $2$-degenerate distance graph in $\R^2$.

In general, one can construct a proper $2$-degenerate distance graph in $\R^3$ as follows: Start with any proper cycle with vertices in $\R^3$, such as a proper triangle (three vertices in general position) or four vertices forming a ``non-rigid" square (no diagonal edges), and at every step attach another proper cycle (or tree) to any of the edges (or vertices) of the graph constructed at the previous step.

\smallskip

\item[4.] A complete graph with vertices $\{v_0,\dots,v_k\}\subseteq\R^{k}$ forms a
proper $k$-degenerate distance graph if and only if $\{v_0,\dots,v_k\}$ are in general position. Another example of a proper $k$-degenerate distance graph in $\R^{k}$ is the ``$k$-dimensional grid" with vertices $V=\{(i_1,\dots,i_k)\,:\,0\leq i_1,\dots,i_k\leq n\}\subseteq\R^k$ and edges $E=\{((i_1,\dots,i_k),(i'_1,\dots,i'_k))\,:\, |i_1-i'_1|+\dots+|i_k-i'_k|=1\}$.

More generally, one can construct a proper $k$-degenerate distance graph in $\R^{k+1}$ as follows: Start with any known proper $k$-degenerate distance graph with vertices in $\R^{k+1}$ and at every step attach another proper $\ell$-degenerate distance graph with $\ell\leq k$ to any of the faces, edges, or vertices of the graph constructed at the previous step.
\end{enumerate}

\medskip

\noindent
\emph{Remark on the Sharpness of the dimension condition in Theorem \ref{infinite}}

Let $e_1,\ldots,e_k$ be the standard basis vectors of $\R^k$ and $\Delta_+$ and $\Delta_-$ denote the complete graphs with vertices $\{0,e_1,e_2,\ldots,e_k\}$ and $\{0,-e_1,e_2,\ldots,e_k\}$ respectively. It is clear that $\Gamma=\Delta_+\cup\Delta_-$ then defines a proper $k$-degenerate distance graph with the property that any isometric copy of $\lambda\cdot\Gamma$ in $\R^k$ must contain three collinear points, i.e. a copy of $\{-\lambda e_1,0,\lambda e_1\}$ obtained by a translation and a rotation. It was shown in \cite{B} that there are sets of positive upper density $A\subseteq \R^k$, for any $k$, which do not contain such configurations for all large $\lambda$. This example shows the sharpness of the dimension condition $d\geq k+1$ in Theorem \ref{infinite}.

%\newpage

\section{A Counting Function and Generalized von-Neumann inequality}\label{cf}
Let  $\Gamma=\Gamma(V, E)$ be a fixed  proper $k$-degenerate distance graph with vertex set $V=\{v_0,v_1,\dots,v_n\}$ with $v_0=0$ in $\R^d$ with $d\geq k+1$.
As our arguments are analytic, we need to define a measure on the configuration space of all isometric copies of $\Gamma$. For each $(v_i,v_j)\in E$ let $t_{ij}=|v_i-v_j|^2$. The configuration space of all isometric copies of $\Gamma$, with the vertex $v_0$ remaining fixed at $0$, namely
% Notice that a distance graph $\Gamma'=\{x_0,x_1,\dots,x_n\}$ with $x_0=0$ is isometric to $\lm^{1/2}\cdot\Gamma$ if and only of $|x_i-x_j|^2=\lm t_{ij}$, for all $i,j$ for which $(v_i,v_j)\in E$. Thus
 \be
 S_{\Gamma}:=\{(x_0,x_1,\dots,x_n)\in\R^{d(n+1)}\,: x_0=0 \ \text{and} \ |x_i-x_j|^2= t_{ij} \ \text{for all $i,j$ for which $(v_i,v_j)\in E$}\}
 \ee
is clearly a real subvariety. Since $\Gamma$ is proper, there exists a points $(x_0,x_1,\dots,x_n)\in S_{\Gamma}$ such that for all $1\leq j\leq n$, the sets $\overline{X}_j:=x_j\cup X_j$ where $X_j:=\{x_i\,:\,v_i\in V_j\}$ are in general position and each of the spheres
\be
S_{j}:=\{x\in\R^d\,:\, |x-x_i|^2=t_{ij} \ \text{for all $x_i\in X_j$}\}
\ee
 have dimension $d-|X_j|$, which is at least $1$ if our distance graph $\Gamma$ is $k$-degenerate and $d\geq k+1$.

 It is easy to see that the radius $r_{j}$ of $S_{j}$ %satisfies $r_{j,\lm}=\lm^{1/2}r_{j}$ and
is positive and equals the distance from $x_j$ to the affine subspace $\text{span }X_j$.
If
$X_j=\{x_{i_1},\dots,x_{i_\ell}\}$, then the fact that $\overline{X}_j=\{x_{i_1},\dots,x_{i_\ell},x_j\}$ is in general position ensures that the volume of the $\ell$-dimensional fundamental parallelepiped determined by the vectors $\{x_j-x_{i_1},\dots,x_j-x_{i_\ell}\}$ is non-zero.
Since this volume is equal to the square root of the determinant of the inner product matrix
\[\det\{(x_j-x_{i_{m_1}})\cdot(x_j-x_{i_{m_2}})\}_{1\leq m_1,m_2\leq \ell}\]
it follows that
\be
r_{j}=\sqrt{\dfrac{\det\{(x_j-x_{i_{m_1}})\cdot(x_j-x_{i_{m_2}})\}_{1\leq m_1,m_2\leq \ell}}{\det\{(x_{i_\ell}-x_{i_{m_1}})\cdot(x_{i_\ell}-x_{i_{m_2}})\}_{1\leq m_1,m_2\leq \ell-1}}}
\ee
and hence that $r_j$ in fact depends continuously on $X_j$.

An important consequence of the discussion above that there exists a constant $r_\Gamma>0$ and compactly supported functions $\eta_j\in C^\infty(\R^d)$ with $0\leq\eta_j\leq1$ such that   for all $(x_0,x_1,\dots,x_n)\in S_{\Gamma}$ with $x_j\in\supp\eta_j$, the corresponding spheres $S_j$ will all have radius $r_j\geq r_\Gamma$.

%For  $0\leq m\leq n-1$ we let $\Gamma_{m}$ denote the subgraph of $\Gamma$ restricted to the vertex set $V\setminus\{v_{m+1},\dots,v_n\}$, i.e. obtained from $\Gamma$ by removing the vertices $v_{m+1},\dots,v_n$ and all edges emanating from them, and set $\Gamma_{n}:=\Gamma$.

\begin{defn}[Localized Counting Function]\label{defn}
For any $0<\lm\ll1$
%, $0\leq m\leq n$, 
and functions
\[f_0,f_1,\dots,f_{n}:[0,1]^{d}\to\R\] with $d\geq k+1$ we  define
\be\label{LocalCount}
T_{\Gamma}(f_0,f_1,\dots,f_{n})(\lm)
=\iint\cdots\int f_0(x)f_1(x-\lm x_1)\cdots f_n(x-\lm x_n)\,d\mu_n(x_n)\cdots d\mu_1(x_1)\,dx
\ee
where $d\mu_j(x_j)=\eta_j(x_j)\,d\sigma_j(x_j)$ and $\sigma_{j}$ denotes the normalized surface measure on $S_j$.
\end{defn}

Note that if $A\subseteq[0,1]^d$ with  $d\geq k+1$ and $T_{\Gamma}(1_A,1_A,\dots,1_A)(\lm)>0$, then $A$ must contain a point configuration $\Gamma'=\{x,x+\lm x_1, \dots, x+\lm x_n\}$ with $(0,x_1,\dots, x_n)\in S_\Gamma$ and hence an isometric copy of $\lambda\cdot\Gamma$.

The key to showing that $T_{\Gamma}(1_A,1_A,\dots,1_A)(\lm)$ is positive for certain sets $A$ is to estimate (\ref{LocalCount}) in terms of a suitable uniformity norm localized to a scale $L$ (related to $\lambda$).

\begin{defn}[$U^1(L)$-norm]
For $0<L\ll1$ and functions $f:[0,1]^d\to\R$ we define
\bee
\|f\|_{U^1(L)}=\|f*\vp_L\|_2
\eee
where $\vp_L(x)=L^{-d}\,\vp(L^{-1}x)$ with $\vp=1_{[-1/2,1/2]^d}$.
\end{defn}

Note that if  $A\subseteq[0,1]^d$ with $\A=|A|>0$ and we define
$
f_A:=1_A-\A1_{[0,1]^d}
$,
then
\be\label{relate}
\|f_A\|^2_{U^1(L)}=\int\limits_{\R^d}\left|\frac{|A\cap(t+Q_{L})|}{|Q_{L}|}-|A|\,\right|^2\,dt,
\ee
where $Q_L=[-L/2,L/2]^d$.

\smallskip
Evidently the $U^1(L)$-norm is measuring the mean-square uniform distribution of $A$ on scale $L$. The engine that drives our approach to Theorem \ref{infinite} is the following
%Specifically if $A$ is $(\VE,L)$-uniformly distributed, then $\|f_A\|_{U^1(L)}\leq2\VE$  provided $0<L\ll\VE$.

%\newpage

\begin{propn}[Generalized von-Neumann]\label{GvN}
Let $0<\VE, \lm\ll 1$. For any $L\leq \VE^6\lm$, $0\leq m\leq n$ and functions
\[f_0,f_1,\dots,f_{m}:[0,1]^{d}\to[-1,1]\]
we have that
\bee
\left|T_{\Gamma}(f_0,f_1,\dots,f_{m},1,\ldots,1)(\lm)\right|\leq \|f_m\|_{U^1(L)}+O_{\Gamma}(\VE).
\eee
\end{propn}

Here $1$ stands for the indicator function of the unit cube $[0,1]^d$ and $O_{\Gamma}(\VE)$ means a quantity bounded by $C_\Gamma\,\VE$ with a constant $C_\Gamma$ depending only on $\Gamma$. We will also use the notation $f\ll_\Gamma g$ to indicate that $|f|\leq c_\Gamma\, g$ with a constant $c_\Gamma >0$ \emph{sufficiently small} for our purposes.

The above proposition immediately implies the following result for uniformly distributed sets from which we will deduce both parts of Theorem \ref{infinite} in Section \ref{proofs} below.

\begin{cor}\label{Cor}

Let  $\Gamma$ be a proper $k$-degenerate distance graph with $n+1$ vertices in $\R^d$ with $d\geq k+1$.

Let  $\A\in(0,1)$ and
$0<\lm\leq\VE\ll_\Gamma \A^{n+1}$.
If $A\subseteq[0,1]^d$ with $|A|=\A$ satisfies $\|f_A\|_{U^1(\VE^6\lm)}\ll\VE$
%=\int\limits_{\R^d}\left|\frac{|A\cap(t+Q_{L})|}{|Q_{L}|}-|A|\,\right|^2\,dt\ll|A|^{n+1},\]
%where $f_A:=1_A-\A1_{[0,1]^d}$
, then
\[T_{\Gamma}(1_A,1_A,\dots,1_A)(\lm)\geq \dfrac{c_{0}}{2}\,\A^{n+1}\]
where
\[c_{0}=\iint\cdots\int \,d\mu_n(x_n)\cdots d\mu_1(x_1)\,dx\]
\end{cor}

\begin{proof}
The result follows immediately from Proposition \ref{GvN} since
\[T_{\Gamma}(1_A,\dots,1_A)(\lm)=c_{0}\,\A^{n+1}+\sum_{m=0}^n \A^{n-m}\,T_{\Gamma}(\underbrace{1_A,\dots,1_A}_{\text{$m$ copies}},f_A,1,\ldots,1)(\lm)\]
where $f_A=1_A-\A1_{[0,1]^d}$.
\end{proof}

We conclude this section with the proof of Proposition \ref{GvN}.
\begin{proof}[Proof of Proposition \ref{GvN}]
Fix $0\leq m\leq n$.  We have
\begin{align*}
&\left|T_{\Gamma}(f_0,f_1,\dots,f_{m},1,\ldots,1)(\lm)\right|\\
&\quad\quad\quad\quad\leq\int\cdots\int\left(\int\Bigl|  \int f_m(x-\lm x_m)\,c_{m+1}(x_1,\ldots,x_m)\,d\mu_m(x_m) \Bigr|\,dx\right)
d\mu_{m-1}(x_{m-1})\cdots d\mu_1(x_{1})
\end{align*}
where 
\be\label{cm}
c_{m+1}(x_1,\dots,x_{m})=\int\cdots\int \,d\mu_n(x_n)\cdots d\mu_{m+1}(x_{m+1})\ee

\noindent
if $0\leq m\leq n-1$ and $c_{n+1}=1$. 
It follows from an application of Cauchy-Schwarz and Plancherel that
\begin{align}\label{CS}
\left|T_{\Gamma}(f_0,f_1,\dots,f_{m},1,\ldots,1)(\lm)\right|^2 \leq
\int|\widehat{f_m}(\xi)|^2 I_m(\lm\,\xi)\,d\xi
\end{align}
where
\be\label{I}
I_m(\xi)=\int\cdots\int
\bigl|\widehat{c_{m+1}\mu_m}(\xi)\bigr|^2\,
d\mu_{m-1}(x_{m-1})\cdots d\mu_1(x_{1})
\ee
with
\[\widehat{c_{m+1}\mu_m}(\xi)=\int c_{m+1}(x_1,\ldots,x_m)\eta_m(x_m)\,e^{-2\pi ix_m\cdot  \xi}\,d\sigma_m (x_m)\]
if   $2\leq m\leq n$ and $I_1=|\widehat{c_2\mu_1}|^2$.
In light of the trivial uniform bound $0\leq I_m(\xi)\leq1$ and the fact that \[\|f_m\|_{U^1(L)}^2=\int|\widehat{f_m}(\xi)|^2|\widehat{\vp}(L\xi)|^2\,d\xi\]
it suffices to establish that
\be\label{key}
I_m(\lm\xi) (1-\widehat{\vp}(L\xi)^2)=O_\Gamma (\VE^{2}).
\ee

Since $0\leq \widehat{\vp}(\xi)^2\leq 1$ for all $\xi\in\R^d$ and $\widehat{\vp}(0)=1$ it follows that
$
0\leq 1-\widehat{\vp}(L\xi)^2\leq\min\{1,4\pi L|\xi|\}.
$
The uniform bound (\ref{key}) thus reduces to establishing the decay estimate
\be\label{Ibound}
I_m(\xi)\leq  \min\{1,C_\Gamma\,|\xi|^{-1/2}\}
\ee
since this would in turn imply that
\bee
I_m(\lm\xi) (1-\widehat{\vp}(L\xi)^2)\leq C_\Gamma\min\{(\lm |\xi|)^{-1/2},\VE^6\lm|\xi|\}\leq C_\Gamma\,\VE^{2}
\eee
whenever $L\leq \VE^6\lm$.

To establish (\ref{Ibound}) we will use the fact that in addition to being trivially bounded by 1, the Fourier transform of $c_{m+1}\mu_m$ also decays for large $\xi$ in certain directions, specifically
\be\label{decay}
|\widehat{c_{m+1}\mu_m}(\xi)|\leq \min\{1, (r_\Gamma\cdot(\text{dist}(\xi,\text{span }X_m))^{-1/2}\}
\ee
uniformly over all  $x_1,\dots,x_{m-1}$ with $x_j\in\supp\eta_j$.
This estimate is an easy consequence of the well-known asymptotic behavior of the Fourier transform of the measure on the unit sphere $S^{d-|X_m|}\subseteq\R^{d-|X_m|+1}$ induced by Lebesgue measure, see for example \cite{Stein}.

Using the fact that the measure
$d\sigma_{m-1}(x_{m-1})\cdots d\sigma_1(x_{1})$ is clearly invariant under the rotations
\[(x_1,\dots,x_m)\to(Ux_1,\dots,Ux_m),\] for any $U\in SO(d)$, together with (\ref{decay}) and the fact that $0\leq\eta_j\leq1$ for $1\leq j\leq m$, then gives
\begin{align*}
I_m(\xi)%&\leq C\int\cdots\int (1+ r_n\cdot\text{dist}(\xi,\text{span }X_n))^{-1}\,d\mu_{n-1}(x_{n-1})\cdots d\mu_1(x_{1})\\
&\leq C\int\cdots\int (1+ r_\Gamma\cdot\text{dist}(\xi,\text{span }X_m))^{-1}\,d\sigma_{m-1}(x_{m-1})\cdots d\sigma_1(x_{1})\\
&= C\int\cdots\int \int_{SO(d)}(1+ r_\Gamma\cdot\text{dist}(\xi,\text{span }UX_m))^{-1}\,d\mu(U)\,d\sigma_{m-1}(x_{m-1})\cdots d\sigma_1(x_{1})\\
&= C\int\cdots\int \int_{S^{d-1}}(1+ r_\Gamma|\xi|\cdot\text{dist}(y,\text{span }X_m))^{-1}\,d\sigma(y)\,d\sigma_{m-1}(x_{m-1})\cdots d\sigma_1(x_{1})
\end{align*}
where $\sigma$ denote normalized measure on the unit sphere $S^{d-1}$ in $\R^d$  induced by Lebesgue measure.
Estimate (\ref{Ibound}) then follows from the easy observation that the inner integral above satisfies the uniform estimate
\be
\int_{S^{d-1}}(1+ r_\Gamma|\xi|\cdot\text{dist}(y,\text{span }X_m))^{-1}\,d\sigma(y)=O((1+r_\Gamma|\xi|)^{-1/2}).\qedhere
\ee
\end{proof}

\section{Proof of Theorem \ref{infinite}}\label{proofs}

We will deduce Theorem \ref{infinite} from Corollary \ref{Cor} by localizing to cubes on which our set in suitably uniformly distributed. In the case of Part (i) this is achieved as a direct consequence of the definition of upper Banach density, while for Part (ii) this is achieved via an energy increment argument.
%These arguments are presented in Sections \ref{newreduction} and \ref{22} below respectively.

\subsection{Direct Proof of Part (i) of Theorem \ref{infinite}}\label{newreduction}
Let $\VE>0$ and $A\subseteq\R^d$  with $\D^*(A)>0$.\smallskip
\comment{Recall that the \emph{upper Banach density} $\D^*$ of a measurable set $A\subseteq\R^d$ is defined by
\be\label{BD}
\D^*(A)=\lim_{N\rightarrow\infty}\sup_{x\in\R^d}\frac{|A\cap(x+Q_N)|}{|Q_N|},\ee
where $|\cdot|$ denotes Lebesgue measure on $\R^d$ and $Q_N$ denotes the cube $[-N/2,N/2]^d$.}
%\subsection{Proposition \ref{Propn0} implies Theorem \ref{Thm0}}\label{Propn0 implies Thm0}
%\[\D^*(A):=\lim_{N\rightarrow\infty}\sup_{x\in\R^d}\frac{|A\cap(x+Q_N)|}{|Q_N|}>0.\]
%where $|\cdot|$ denotes Lebesgue measure on $\R^2$ and $Q_N$ the cube $[-N/2,N/2]^2$.

 The following two facts follow immediately from the definition of upper Banach density, see (\ref{BD}):
\begin{itemize}
\item[(i)] There exist $M_0=M_0(A,\VE)$ such that for all $M\geq M_0$ and all $t\in\R^d$
\[\frac{|A\cap(t+Q_M)|}{|Q_M|}\leq(1+\VE^4/3)\,\D^*(A).\]
\item[(ii)] There exist arbitrarily large $N\in\R$ such that
\[\frac{|A\cap(t_0+Q_N)|}{|Q_N|}\geq(1-\VE^4/3)\,\D^*(A)\]
for some $t_0\in\R^d$.
\end{itemize}

Combining (i) and (ii) above we see that for any $\lm\geq \lm_0:=\VE^{-6}M_0$, there exist $N\geq\VE^{-6}\lm$ and $t_0\in\R^d$ such that
\bee
\frac{|A\cap(t+Q_{\VE^6\lm})|}{|Q_{\VE^6\lm}|}\leq(1+\VE^4)\frac{|A\cap(t_0+Q_N)|}{|Q_N|}
\eee
for all $t\in\R^d$.
%Theorem \ref{BourSimp} thus reduces, via a rescaling of $A\cap(t_0+Q_N)$ to a subset of $[0,1]^d$, to Proposition \ref{Propn00} below.
Consequently, Theorem \ref{infinite} reduces, via a rescaling of $A\cap(t_0+Q_N)$ to a subset of $[0,1]^d$,
to establishing that if
$\Gamma$ is a proper $k$-degenerate distance graph,
%\begin{propn}\label{Propn00}
%Let
$0<\lm\leq\VE\ll1$ and $A\subseteq[0,1]^d$ is measurable with $|A|>0$ and the property  that
\bee
\frac{|A\cap(t+Q_{\VE^6\lm})|}{|Q_{\VE^6\lm}|}\leq (1+\VE^4)\,|A|
\eee
for all $t\in\R^d$,  then $A$ contains an isometric copy of $\lm\cdot\Gamma$.

Now since $A\cap(t+Q_{\VE^6\lm})$ is only supported in $[-\VE^6\lm,1+\VE^6\lm]^d$ and
\bee\label{Q-QL}
|A|=\int_{\R^d}\frac{|A\cap(t+Q_{\VE^6\lm})|}{|Q_{\VE^6\lm}|}\,dt
%=\int_{[0,1]^d}\frac{|A\cap(t+Q_{\VE^4\lm})|}{|Q_{\VE^4\lm}|}\,dt+O(\VE^4|A|),
\eee
it easily follows that
%from which one can easily deduce that
\bee\label{most t's}
\Bigl|\Bigl\{t\in\R^d\,:\, 0<\frac{|A\cap(t+Q_{\VE^6\lm})|}{|Q_{\VE^6\lm}|}\leq (1-\VE^2)\,|A|\Bigr\}\Bigr|=O(\VE^2)
\eee
and hence that
\[
\|f_A\|^2_{U^1(\VE^6\lm)}=\int\limits_{\R^d}\left|\frac{|A\cap(t+Q_{\VE^6\lm})|}{|Q_{\VE^6\lm}|}-|A|\,\right|^2\,dt=O(\VE^2).
\]
The result thus follows from Corollary \ref{Cor} above provided $\VE\ll \D^*(A)^{n+1}$. \qed

\subsection{Proof of Part (ii) of Theorem \ref{infinite}}\label{22}

\begin{lem}[Localization Principle]\label{local}
Let $A\subseteq[0,1]^d$ with $d\geq k+1$ and $|A|=\A>0$.

Let $\VE>0$ and  $\VE^7\gg L_1\gg L_2\gg\cdots$ be any decreasing sequence with $L_1^{-1}\in\N$ and $L_{j+1}\leq c\,\VE^7L_j$ with $L_{j+1}|L_j$ for all $j\geq1$. If we let $\mathcal{G}_j$ denote the partition of $[0,1]^d$ into cubes of sidelength $L_j$, then
there exists $1\leq j\leq C\VE^{-2}$ such that for all but at most $\VE L_j^{-d}$ of the cubes $Q$ in  $\mathcal{G}_j$ the set $A$ will be %\emph{$(\VE^{1/2},L_{j+1})$-uniformly distributed on $Q$}
uniform distributed on the smaller scale $L_{j+1}$ inside $Q$ in the sense that
\be
\frac{1}{|Q|}\int_{Q}\left|\frac{|A\cap Q\cap(t+Q_{L_{j+1}})|}{|Q_{L_{j+1}}|}-\frac{|A\cap Q|}{|Q|}\,\right|^2\,dt\leq\VE.
\ee
\end{lem}

\smallskip

Before proving Lemma \ref{local} we first show that it, together with Corollary \ref{Cor} (after rescaling), is sufficient to establish Part (ii) of Theorem \ref{infinite}.
Let $\VE\ll_\Gamma\,\A^{n+1}$ and $\{Q_i\}$
%, with $1\leq i\leq L_j^{-d}$,
denote the cubes of sidelength $L_j$ in the partition $\mathcal{G}_j$ of $[0,1]^d$ that we obtain from Lemma \ref{local}. If we then let $A_i=A\cap Q_i$ and set $\A_i=|A\cap Q_i|/|Q_i|$ it follows from Corollary \ref{Cor} (after rescaling) and H\"older's inequality that for any $\lm\in (\VE^{-6}L_{j+1},\VE L_j)$ we have
\be
T_{\Gamma}(1_A,\dots,1_A)(\lm)\geq \sum_{i=1}^{L_j^{-d}}T_{\Gamma}(1_{A_i},\dots,1_{A_i})(\lm)\geq \dfrac{c_{0}}{4} \, L_j^d\sum_{i=1}^{L_j^{-d}} \A_i^{n+1}\geq \dfrac{c_{0}}{4} \, \left(L_j^d\sum_{i=1}^{L_j^{-d}} \A_i\right)^{n+1}\!\!\!\!=\dfrac{c_{0}}{4} \, |A|^{n+1}.
\ee

%%%%%%%%%%%%%%

\begin{proof}[Proof of Lemma \ref{local}]

Let $\{Q_i\}$
denote the cubes of sidelength $L_j$ in  the partition $\mathcal{G}_j$ of $[0,1]^d$ and \[g_j=1_A-\mathbb{E}(1_A\,|\,\mathcal{G}_j)\] where
\[\mathbb{E}(1_A\,|\,\mathcal{G}_j)(x)=\dfrac{|A\cap Q_i|}{|Q_i|}\]
for each $x\in Q_i$.
 If $\|g_j\|_{U^1(L_{j+1})}\geq\VE$, then by definition
\[\int\Bigl|\dfrac{1}{|Q_{L_{j+1}}|}\int_{x+Q_{L_{j+1}}}g_j(y)\,dy\Bigr|^2\,dx\geq c\, \VE^2.\]
It follows that there must exist a $x_0\in[0,1]^d$ for which the shifted grid $x_0+\mathcal{G}_{j+1}$ satisfies
\[\int\Bigl|\mathbb{E}(g_j\,|\,x_0+\mathcal{G}_{j+1})\Bigr|^2\,dx\geq c\, \VE^2\]
from which one can easily conclude that the (unshifted) refined grid $\mathcal{G}_{j+2}$ satisfies
\be\label{unshifted}
\int\Bigl|\mathbb{E}(g_j\,|\,\mathcal{G}_{j+2})\Bigr|^2\,dx\geq c\, \VE^2\ee
provided $L_{j+2}\ll \VE^2 L_{j+1}$. By orthogonality, it follows immediately from (\ref{unshifted}) and the definition of $g_j$ that
\be
\|\mathbb{E}(1_A\,|\,\mathcal{G}_{j+2})\|_2^2\geq \|\mathbb{E}(1_A\,|\,\mathcal{G}_{j})\|_2^2+c\VE^2
\ee
and hence that there must exist $1\leq j\leq C \VE^{-2}$ such that $\|g_j\|_{U^1(L_{j+1})}\leq\VE$ from which it follows that
%Recall that $\|g_j\|_{U^1(L_{j+1})}\leq\VE$ implies that
%\[\int\Bigl|\dfrac{1}{|Q_{L_{j+1}}|}\int_{x+Q_{L_{j+1}}}g_j(y)\,dy\Bigr|^2\,dx\leq \VE^2\]
%and he
\[\sum_{i=1}^{L_j^{-d}}\int\Bigl|\dfrac{1}{|Q_{L_{j+1}}|}\int_{x+Q_{L_{j+1}}}(1_{A_i}-\A_i1_{Q_i})(y)\,dy\Bigr|^2\,dx\leq C \VE^2\]
provided $L_{j+1}\ll\VE^2 L_j$.
\end{proof}

%We have in fact established a so-called  Koopman-von Neumann decomposition of $1_A$, namely we have identified for any $\VE>0$ a partition $\mathcal{G}$ of $[0,1]^d$ into cubes of sidelength at least

%\newpage

\section{A Second Proof of Theorem \ref{infinite}}

We conclude  by presenting a second proof of Theorem \ref{infinite} which is closer in spirit to Bourgain's original proof of Theorem \ref{BourSimp}. We include this in order to highlight the simplicity and directness of our approach to Part (i) of Theorem \ref{infinite} above, but also to emphasize that our approach to Part (ii) of Theorem \ref{infinite} is in essence a physical space reinterpretation of Bourgain original approach.

\comment{
\begin{thm2}
Let $\Gamma$ be a proper $k$-degenerate distance graph and $0<\A<1$.
There exists $c(\A,\Gamma)>0$ such that any $A\subseteq[0,1]^d$ with $d\geq k+1$ and $|A|=\A$ will contain an isometric copy of $\lm\cdot\Gamma$ for all $\lambda$ in some interval of length at least $c(\A,\Gamma)$.
\end{thm2}
}

\subsection{Reducing Theorem \ref{infinite} to a Dichotomy between Randomness and  Structure}

Let $\Gamma$
be a proper $k$-degenerate distance graph in $[0,1]^d$ with $d\geq k+1$.
As we shall see, Theorem \ref{infinite} is an immediate consequence of the following proposition which reveals that if $A\subseteq[0,1]^d$ has positive measure but does not contain an isometric copy of $\lm\cdot\Gamma$ for all $\lm$ in a given interval, then this ``non-random" behavior is detected by the Fourier transform of the characteristic function of $A$ and results in ``structural information''%in the form of a ``Fourier-deviation from randomness"
, specifically a concentration of its $L^2$-mass on appropriate annuli.

\begin{propn}[Dichotomy]\label{dic}
Let $\Gamma=\Gamma(V,E)$
%\{0,v_1,\dots,v_n\}
 be a proper $k$-degenerate distance graph in $[0,1]^d$ with $d\geq k+1$.
%Let $\delta>0$, $0<\VE\ll\delta^3$, and $0<a\leq\VE^{-4} b\leq \VE^2$.

If $A\subseteq[0,1]^d$ with $|A|>0$, $0<a\leq b\ll \VE^{4}$ with $0<\VE\ll_\Gamma|A|^{n+1}$, and $A$ does \underline{not} contain an isometric copy of $\lm\cdot\Gamma$ for some $\lambda$ in $[a,b]$, then
%\be\label{i}\int%_{[0,1]^d}
%\int_{SO(d)}1_A(x)1_A(x-\lm\cdot U(v_1))1_A(x-\lm\cdot U(v_2))\,d\mu(U)\,dx=0\ee
%for some $\lm\in\mathcal{I}=[a,b]$, where $\mu$ denotes the unique Haar measure on $SO(d)$, then
\be\label{ii} \int\limits_{\VE^2/b\leq|\xi|\leq1/\VE^{2}a} |\widehat{1_A}(\xi)|^2\,d\xi \gg |A|^{2n+2}\ee
with the implied constant above independent of $a$, $b$, and $\VE$.
\end{propn}

\begin{proof}[Proof that Proposition \ref{dic} implies Theorem \ref{infinite}] We shall first establish Part (ii) of Theorem \ref{infinite}, so we start by
letting $A\subseteq[0,1]^d$ with $|A|>0$.
 For any fixed $0<\VE\ll_\Gamma|A|^{n+1}$, let $\{\mathcal{I}_j\}_{j=1}^{J(\VE)}$ denote a sequence of intervals with $\mathcal{I}_j:=[a_j,b_j]$ satisfying
\be\label{disjoint}
b_{j+1}\ll\VE^4a_j
\ee
and $b_1\ll\VE^4$
with the property that for each $1\leq j\leq J(\VE)$ there exists a $\lm\in\mathcal{I}_j$ such that
 \be
x+\lm\cdot U(\Delta)\nsubseteq A\ee
for all $x\in A$ and $U\in SO(d)$. Proposition \ref{dic}, together with (\ref{disjoint}),  would then imply that
\be
J(\VE)\,\VE^2\leq\sum_{j=1}^{J(\VE)}
\int\limits_{\VE^2/b_j\leq|\xi|\leq1/\VE^{2}a_j} |\widehat{1_A}(\xi)|^2\,d\xi\leq\int |\widehat{1_A}(\xi)|^2\,d\xi
\ee
a contradiction if $J(\VE)\gg\VE^{-2}$ since by Plancherel we know that
$\int |\widehat{1_A}(\xi)|^2\,d\xi=|A|\leq1.$

\smallskip

To establish Part (i) of Theorem \ref{infinite} with this approach we will argue indirectly and thus suppose that $A\subseteq \R^d$ is a set with $\delta^*(A)>0$ for which the conclusion of Part (i) of Theorem \ref{infinite} fails to hold, namely that there exist arbitrarily large $\lm\in\R$ for which $A$ does not contain an isometric copy of $\lm \cdot\Gamma$.

We now let  $0<\A<\delta^*(A)$, $0<\VE\ll_\Gamma\A^{n+1}$, and fix $J\gg\VE^{-2}$ as above. By our indirect assumption we can choose a sequence $\{\lm_j\}_{j=1}^J$ with the property that
$\lm_{j+1}\ll\VE^4\lm_j$ for all $1\leq j\leq J-1$ and $A$ does not contain an isometric copy of $\lm_j \cdot\Gamma$ for each $1\leq j\leq J$.
It follows from the definition of upper Banach density that exist $N\in\R$ with $N\gg\lm_1$ and $t_0\in\R^d$ for which
\[\frac{|A\cap(t_0+Q_N)|}{|Q_N|}\geq\alpha.\]
Rescaling $A\cap(t_0+Q_N)$ to a subset of $[0,1]^d$ and arguing as in the proof of Part (ii) above but this time with $b_j=\lm_j/N$ again leads to a contradiction.
\end{proof}

%This proposition reveals that if $A$ does not contain an isometric copies of $\lm\cdot\Delta$ for all $\lm$ in a given interval, then this ``non-random" behavior is ``detected" by the Fourier transform of the characteristic function of $A$ and results in ``structural information'', specifically a concentration of its $L^2$-mass on appropriate annuli.

%\bigskip

\subsection{Proof of Proposition \ref{dic}}

Let $f=1_A$ and $\Gamma=\{0,v_1,\dots,v_n\}$ be a fixed proper $k$-degenerate distance graph.
%If $A$ does \underline{not} contain an isometric copy of $\lm\cdot\Gamma$ for some $\lambda$ in $[a,b]$, then in particular the ``localized counting function", see (\ref{LocalCount}), satisfies
%Recall that in (\ref{LocalCount}) we defined
%\be
%T_{\Gamma_m}(f_0,f_1,\dots,f_{m})(\lm)
%=\iint\cdots\int f_0(x)f_1(x-\lm x_1)\cdots f_m(x-\lm x_m)\,d\mu_m(x_m)\cdots d\mu_1(x_1)\,dx
%\ee
%for $0\leq m\leq n$ where $d\mu_j(x_j)=\eta_j(x_j)\,d\sigma_j(x_j)$% and $\sigma_{j}$ denotes the normalized surface measure on $S_j$
%, see Section \ref{cf}.
We will utilize the existence of a suitably \emph{smoothed} version of $f$ with the certain properties, specifically
\begin{lem}\label{gprop}
For any $\VE>0$ there exists a function $g:\R^d\to(0,1]$, an appropriate smoothing of $f$, such  that
\be\label{1}
|g(x-\lm z)-g(x)|\ll\VE
%\int\int\bigl|\psi_{t}(x-\lm x_1)-\psi_{t}(x)\bigr|\,d\sigma(x_1)\,dx \ll \eta
\ee
uniformly in $x\in[0,1]^d$ and $|z|\leq 1$. Moreover, if $\VE\ll|A|^{n+1}$, then
\be\label{gmain}
\int f(x)g(x)^n\,dx\gg |A|^{n+1}.
\ee
\end{lem}
The proof of Lemma \ref{gprop} is straightforward and presented in Section \ref{3.2} below. %, follows by taking $g$ to be an appropriately smoothed version of $f$.
Assuming for now the existence of  a function $g$ with
property (\ref{1}) it follows that
\be\label{sum}
T_{\Gamma}(f,f,\dots,f)(\lm)=\int f(x)g(x)^n\,dx+\sum_{m=1}^n T_{\Gamma}(fg^{n-m},f,\dots, f,f-g, \!\!\!\underbrace{1,\dots,1}_{\text{$n-m$ copies}}\!\!\!)(\lm)+O(n\VE)\ee
where, as in (\ref{LocalCount}) in Section \ref{cf}, we define
\[T_{\Gamma}(f_0,f_1,\dots,f_{n})(\lm)
=\iint\cdots\int f_0(x)f_1(x-\lm x_1)\cdots f_n(x-\lm x_n)\,d\mu_n(x_n)\cdots d\mu_1(x_1)\,dx\]
with $d\mu_j(x_j)=\eta_j(x_j)\,d\sigma_j(x_j)$ and $\sigma_{j}$ denotes the normalized surface measure on $S_j$.

If $A$ does \underline{not} contain an isometric copy of $\lm\cdot\Gamma$ for some $\lambda$ in $[a,b]$, then it clearly follows that \[T_{\Gamma}(f,f,\dots,f)(\lm)=0.\]
In light of (\ref{gmain}) and (\ref{sum}) it follows that if  $\VE\ll|A|^{n+1}/n$
 then %either
% \be
%\int\left|\int [f-g](x-\lm x_1)\,d\mu_{1}(x_1)\right|dx \gg |A|^{n+1}.
%\ee
%or
there must exist $1\leq m\leq n$ such that
\be
\int\cdots\int\left(\int\left|\int [f-g](x-\lm x_m)\,c_{m+1}(x_1,\ldots,x_m)\,d\mu_{m}(x_m)\right|dx\right)\,d\mu_{m-1}(x_{m-1})\cdots d\mu_1(x_{1}) \gg |A|^{n+1}
\ee
with $c_{m+1}$ defined as before in equation (\ref{cm}) above.
It then follows from an application of Cauchy-Schwarz and Plancherel that
\be\label{wf-g}
\int |\widehat{f}(\xi)-\widehat{g}(\xi)|^2\,I_m(\lm\xi)\,d\xi \gg |A|^{2n+2}
\ee
with $I_m$ again defined as before in equation (\ref{I}) above.
%where
%\vspace{-1pt}
%\be\label{I}
%I_m(\xi)=\int\cdots\int
%|\widehat{\mu_{m}}(\xi)|^2\,d\mu_{m-1}(x_{m-1})\cdots d\mu_1(x_{1}).
%\ee
The fact that $g$ will be taken to be a sufficient smoothing of $f$ ensures that its Fourier transform satisfies
\be\label{gfourier}
|\widehat{f}(\xi)-\widehat{g}(\xi)|\leq \VE\,|\widehat{f}(\xi)|
\ee
provided $|\xi|\leq\VE^2b^{-1}$, see Section \ref{3.2} below.
This, together with the fact that $I_m(\xi)$ is bounded by 1 uniformly in $\xi$, and Plancherel, ensures that (\ref{wf-g}) implies
%We have thus established, under the assumption that there exists a function $g$ with properties (\ref{1}), (\ref{gmain}) and (\ref{gfourier}), that if hypothesis (\ref{i}) holds and  $\VE\ll|A|^3$, then
\be\label{MM}
\int\limits_{\VE^2/b\leq|\xi|} |\widehat{f}(\xi)|^2\,I_m(\lm\xi)\,d\xi \gg |A|^{2n+2}.
\ee

Estimate (\ref{ii}), and hence Proposition \ref{dic}, then follows easily from estimate (\ref{MM}) and our previously established estimates for $I_m$, namely (\ref{Ibound}).

%The following proposition, together with the ``Fourier decay estimates" in Lemma \ref{E} below, clearly suffice to establish Proposition \ref{dic}.

%\begin{propn}\label{M} Under the hypothesis of Proposition \ref{dic} it follows that \be\int\limits_{\VE^2/b\leq|\xi|} |\widehat{1_A}(\xi)|^2\,(|\widehat{\sigma}(\lm \xi)|+I(\lm\xi))\,d\xi \gg |A|^3
%\ee
%where
%\be\label{I}
%I(\xi)=\int
%|\widehat{\sigma_{x_1}}(\xi)|\,d\sigma(x_1).
%\ee
%\end{propn}

%Recall that for any complex-valued Borel measure $\nu$, such as $d\nu=1_A\,dx$, $d\nu=d\sigma$, or $d\nu=d\sigma_{x_1}$, its Fourier transform
%\be
%\widehat{\nu}(\xi)=\int_{\R^d}e^{-2\pi i x\cdot\xi}\,d\nu(x)\ee
%defines a $1$-bounded continuous function.

%Indeed, it should be clear that Proposition \ref{M} together with the following  ``Fourier decay estimates" is sufficent to establish Proposition \ref{dic}.
%\begin{lem}[Fourier decay estimates]\label{E}
%If $d\geq3$, then
%$
%|\widehat{\sigma}(\xi)|\leq \min\{1,C|\xi|^{-1}\}
%$
%and
%\be\label{bound}
%I(\xi)\leq\min\{1,C|\xi|^{-1/2}\}.%\ll\eta^{(d-j)/2}\ll\eta^{1/2}
%\ee
%\end{lem}

%The proof of Lemma  \ref{E} is presented in Section \ref{3.1} below.

%\subsection{Proof of Lemma \ref{gprop}}\label{3}

\subsection{A smooth cutoff function and  Proof of Lemma \ref{gprop}}\label{3.2}
\subsubsection{A smooth cutoff function}\label{properties}

Let
 $\psi:\R^d\rightarrow(0,\infty)$ be a Schwartz function that satisfies
\[1=\widehat{\psi}(0)\geq\widehat{\psi}(\xi)\geq0\quad\quad\text{and}\quad\quad \widehat{\psi}(\xi)=0 \ \ \text{for} \ \ |\xi|>1.\]
As usual, for any given $t>0$, we define
\be
\psi_{t}(x)=t^{-d}\psi(t^{-1}x).\ee

First we record the trivial observation that
\be\label{=1}\int\psi_{t}(x)\,dx=\int\psi(x)\,dx=\widehat{\psi}(0)=1\ee
as well as the simple, but important, observation that $\psi$ may be chosen so that
\be\label{cutoff2}
\bigl| 1-\widehat{\psi}_{t}(\xi)\bigr|=\bigl| 1-\widehat{\psi}(t\xi)\bigr|\ll\min\{1,t|\xi|\}.%\eta
\ee
%for any given $\eta>0$ provided $t|\xi|\leq\eta$.

Finally we record a formulation, appropriate to our needs, of the fact that for any given small parameter $\VE$, our cutoff function $\psi_{t}(x)$ will be essentially supported where $|x|\leq\VE^{-1}t$ and is approximately constant on smaller scales. More precisely,

\begin{lem}\label{uncertainty} Let $\VE>0$ and $t>0$, then
\be\label{5.3}
\int_{|y|\geq\VE^{-1}t}\psi_{t}(y)\,dy \ll \VE.
\ee
and
\be\label{5.1}
\int\bigl|\psi_{t}(y-\lm z)-\psi_{t}(y)\bigr|\,dy \ll \VE
\ee
uniformly for $|z|\leq1$, provided $t\gg\VE^{-1}\lm$.
\end{lem}

\begin{proof}[Proof of Lemma \ref{uncertainty}]
Estimate (\ref{5.3}) is easily verified using the fact that $\psi$ is a Schwartz function on $\R^d$ as
\[\int_{|y|\geq\VE^{-1}t}\psi_{t}(y)\,dy= \int_{|y|\geq\VE^{-1}}\psi(y)\,dy
\ll \int_{|y|\geq\VE^{-1}}(1+|y|)^{-d-1}\,dy\ll\, \VE.
\]

To verify estimate (\ref{5.1}) we make use of the fact that both
 $\psi$ and its derivative are rapidly decreasing, specifically
\[
\int\bigl|\psi_{t}(y-\lm z)-\psi_{t}(x)\bigr|\,dy
\leq \int\bigl|\psi(y-\lm z/t)-\psi(y)\bigr|\,dy
\ll \frac{\lm}{t} \int(1+|y|)^{-d-1}\,dy\ll \frac{\lm}{t}.\qedhere\]
\end{proof}

\subsubsection{Proof of Lemma \ref{gprop}}
Let $g=f*\psi_{\VE^{-1}b}.$

We first note that estimates (\ref{gfourier}) and  (\ref{1}) follow immediately from (\ref{cutoff2}) and  (\ref{5.1}) respectively. In order to establishing the remaining ``main term'' estimate (\ref{gmain}), we
%, namely  that
%\[\int f(x)g(x)^2\,dx\gg |A|^3\]
%provided $\VE\ll|A|^3$, we
need only establish that if  $\VE\ll|A|^{n+1}$, then
\be\label{k=1}
\int f(x)g(x)\,dx\geq (1-C\VE)\,|A|^2
\ee
for some constant $C>0$, since by H\"older we would then obtain
\[(1-C\VE)^n\,|A|^{2n}\leq\Bigl(\int f(x)g(x)\,dx\Bigr)^n\leq|A|^{n-1}\int f(x)g(x)^n\,dx
\]
from which (\ref{gmain}) clearly follows for sufficiently small $\VE>0$.

To establish (\ref{k=1}) we first note that Parseval, the fact that $0\leq\widehat{\psi}\leq 1$, and a final application of Cauchy-Schwarz gives
\be
\int f(x)g(x)\,dx=\int |\widehat{f}(\xi)|^2\widehat{\psi}(\VE^{-1}b\,\xi)\,d\xi\geq \int |\widehat{f}(\xi)|^2|\widehat{\psi}(\VE^{-1}b\,\xi)|^2\,d\xi=\int g(x)^2\,dx\geq \Bigl(\int_{[0,1]^d} g(x)\,dx\Bigr)^2.
\ee

Establishing (\ref{k=1}) therefore  reduces to showing that if  $\VE\ll|A|^3$, then
\be\label{finally}
\int_{[0,1]^d} g(x)\,dx\geq (1-C\VE)|A|
\ee
for some constant $C>0$.
To establish (\ref{finally}) we use (\ref{=1}) and write %As $f$ is supported on $B_N$ and $\eta\ll\D$ we have
\be
|A|=\int_{\R^d} g(x)\,dx=\int_{[0,1]^d}g(x)\,dx\,\,\,\,+\int\limits_{\{x\in\R^d\,:\,\text{dist}(x,[0,1]^d)\geq\VE^{-2}b\}} \!\!\!\!\!\!\!\!\!\!\!\!g(x)\,dx\,\,\,\,+\int\limits_{\{x\in\R^d\,:\,0<\text{dist}(x,[0,1]^d)<\VE^{-2}b\}} \!\!\!\!\!\!\!\!\!\!\!\! g(x)\,dx.\ee

The fact that $b\leq\VE^{4}$ ensures that
 \be\big|\{x\in\R^d\,:\,0<\text{dist}(x,[0,1]^d)<\VE^{-2}b\}\big|\ll\VE^{2}\ee
and hence, since $\VE\ll|A|$ and $0\leq g\leq 1$, that
\[\int\limits_{\{x\in\R^d\,:\,0<\text{dist}(x,[0,1]^d)<\VE^{-2}b\}} \!\!\!\!\!\!\!\!\!\!\!\! g(x)\,dx\ll\VE^{2}\leq \VE|A|\]
while  (\ref{5.3}) ensures that
\be\int\limits_{\{x\in\R^d\,:\,\text{dist}(x,[0,1]^d)\geq\VE^{-2}b\}} \!\!\!\!\!\!\!\!\!\!\!\!g(x)\,dx \leq |A|\int_{|y|\gg \VE^{-2}b} \psi_{\VE^{-1}b}(y)\,dy \ll \VE |A|\ee
which completes the proof. \qed

\bigskip
%\newpage


\begin{thebibliography}{10}


\bibitem{IO} {\sc M. Bennett, A. Iosevich, K.  Taylor} {\em Finite chains inside thin subsets of $\R^d$}, Anal. PDE. (2015); 9:597-614.

\bibitem{B}
{\sc J. Bourgain}, {\em
A Szemer\'edi type theorem for sets of positive density in $\R^k$},
Israel J. Math. 54 (1986), no. 3, 307--316.

\bibitem{BU}{\sc K. Bulinski},
{\em Spherical Recurrence and locally isometric embeddings of trees into positive density subsets of $\mathbb {Z}^ d$}, to appear in Math. Proc. Cambridge Philos. Soc.



\bibitem{FKW}
{\sc H. Furstenberg, Y. Katznelson and B. Weiss}, {\em
Ergodic theory and configurations in sets of positive density},
Israel J. Math. 54 (1986), no. 3, 307--316.


\bibitem{IP} {\sc A. Iosevich and H. Parshall}, {\em Embedding distance graphs in finite field vector spaces,}  arXiv:1802.06460

\bibitem{Product}
{\sc N. Lyall and A. Magyar}, {\em Product of simplices and sets of positive upper density in $\R^d$}, to appear in Math. Proc. Cambridge Philos. Soc.

%\bibitem{SimpleBourgain}
%{\sc N. Lyall and A. Magyar}, {\em A new proof of Bourgain's theorem on simplices in $\R^d$},
%\texttt{alpha.math.uga.edu/\~\,lyall/SimpleBourgain.pdf}







\bibitem{Stein}
{\sc E. Stein}, {\em
Harmonic Analysis: Real Variable Methods, Orthogonality and Oscillatory Integrals},
Princeton University Press, Princeton, NJ., 1993.

\end{thebibliography}
\end{document}